\documentclass[journal]{IEEEtran}
\usepackage{amsthm}
\newtheorem{theorem}{Theorem}
\newtheorem{lemma}{Lemma}
\theoremstyle{definition}
\newtheorem{definition}{Definition}
\newtheorem{corollary}{Corollary}

\usepackage{algorithm}
\usepackage{algorithmic}

\ifCLASSINFOpdf
\else
   \usepackage[dvips]{graphicx}
   \graphicspath{{../eps/}}
   \DeclareGraphicsExtensions{.eps}
\fi
\ifCLASSOPTIONcompsoc
 \usepackage[caption=false,font=normalsize,labelfont=sf,textfont=sf]{subfig}
\else
 \usepackage[caption=false,font=footnotesize]{subfig}
\fi
\begin{document}
\title{Non-convex Fraction Function Penalty: Sparse Signals Recovered from Quasi-linear Systems}

\author{Angang~Cui,
        Jigen~Peng,
        and~Haiyang~Li
\thanks{The work was supported by the National Natural Science Foundations of China (11771347,11131006,41390450,11761003,11271297) and the Science Foundations of Shaanxi Province of China (2016JQ1029,2015JM1012).}       
\thanks{A. Cui and J. Peng are with the School of Mathematics and Statistics, Xi'an Jiaotong University, Xi'an 710049, China
e-mail: (cuiangang@163.com; jgpengxjtu@126.com).}% <-this % stops a space
\thanks{H. Li is with the School of Science, Xi'an Polytechnic University, Xi'an, 710048, China e-mail: (fplihaiyang@126.com).}% <-this % stops a space
\thanks{Manuscript received , ; revised , .}}


\markboth{Journal of \LaTeX\ Class Files,~Vol.~, No.~, ~}%
{Shell \MakeLowercase{\textit{et al.}}: Bare Demo of IEEEtran.cls for IEEE Journals}



\maketitle

\begin{abstract}
The goal of compressed sensing is to reconstruct a sparse signal under a few linear measurements far less than the dimension of the ambient space of the signal. However, many real-life applications in physics
and biomedical sciences carry some strongly nonlinear structures, and the linear model is no longer suitable. Compared with the compressed sensing under the linear circumstance, this nonlinear compressed sensing is much
more difficult, in fact also NP-hard, combinatorial problem, because of the discrete and discontinuous nature of the $\ell_{0}$-norm and the nonlinearity. In order to get a convenience for sparse signal recovery, we set
most of the nonlinear models have a smooth quasi-linear nature in this paper, and study a non-convex fraction function $\rho_{a}$ in this quasi-linear compressed sensing. We propose an iterative fraction thresholding algorithm to solve the regularization problem $(QP_{a}^{\lambda})$ for all $a>0$. With the change of parameter $a>0$, our algorithm could get a promising result, which is one of the advantages for our algorithm compared with other algorithms. Numerical experiments show that our method performs much better compared with some state-of-art methods.
\end{abstract}


% \fi
\IEEEpeerreviewmaketitle

\section{Introduction}\label{section1}
In compressed sensing (see, e.g., \cite{candes1,dono2}), the problem of reconstructing a sparse signal under a few linear measurements which are far fewer than the dimension of the ambient space of the signal can be
modeled into the following $\ell_{0}$-minimization:
\begin{equation}\label{equ1}
(P_{0})\ \ \ \ \ \min_{x\in \mathcal{R}^{n}}\|x\|_{0}\ \ \mathrm{subject}\ \mathrm{to}\ \ Ax=b
\end{equation}
where $A\in \Re^{m\times n}$ is a $m\times n$ real matrix of full row rank with $m<n$, and $b\in \Re^{m}$ is a nonzero real vector of $m$-dimension, and $\|x\|_{0}$ is the $\ell_{0}$-norm of the real vector $x$, which
counts the number of the non-zero entries in $x$ (see, e.g., \cite{bru3,elad4,the5}). In general, the problem $(P_{0})$ is computational and NP-hard because of the discrete and discontinuous nature of the $\ell_{0}$-norm.

However, many real-life applications in physics and biomedical sciences carry some strongly nonlinear structures \cite{ehler6}, so that the linear model in problem $(P_{0})$ is no longer suitable. We consider a map $A: \Re^{n}\rightarrow \Re^{m}$, which is no longer necessarily linear, and reconstruct a sparse vector $x\in \Re^{n}$ from the measurements $b\in \Re^{m}$ given by
\begin{equation}\label{equ2}
A(x)=b.
\end{equation}
In order to get a convenience for sparse signal recovery, we set most of the nonlinear models have a smooth quasi-linear nature. By this means, there exists a Lipschitz map
\begin{equation}\label{equ3}
F: \Re^{n}\rightarrow \Re^{m\times n}
\end{equation}
such that
\begin{equation}\label{equ4}
A(x)=F(x)x
\end{equation}
for all $x\in \Re^{n}$.

The sparse signals recovered under the quasi-linear case can be mathematically viewed as the following form
\begin{equation}\label{equ5}
(QP_{0})\ \ \ \ \ \min_{x\in \Re^{n}}\|x\|_{0}\ \ \mathrm{subject}\ \mathrm{to}\ \ F(x)x=b.
\end{equation}
Similarly, the quasi-linear compressed sensing is also combinatorial and NP-hard (see, e.g., \cite{ehler6, sigl7}).

The $\ell_{1}$-norm is the most famous convex relaxation (see, e.g., \cite{ehler6, sigl7}), and the minimization for quasi-linear compressed sensing has the following form
\begin{equation}\label{equ6}
(QP_{1})\ \ \ \ \ \min_{x\in \Re^{n}}\|x\|_{1}\ \ \mathrm{subject}\ \mathrm{to}\ \ F(x)x=b.
\end{equation}
for the constrained problem and
\begin{equation}\label{equ7}
(QP_{1}^{\lambda})\ \ \ \ \ \ \min_{x\in \Re^{n}} \Big\{\|F(x)x-b\|_{2}^{2}+\lambda\|x\|_{1}\Big\}
\end{equation}
for the regularization problem, where $\|x\|_{1}=\sum_{i=1}^{n}|x_{i}|$ is the $\ell_{1}$-norm of vector $x$.

In problem $(QP_{1})$, many excellent theoretical works (see, e.g., \cite{ehler6, sigl7}) have shown that the $\ell_{1}$-norm minimization can really make an exact recovery in some specific conditions.
In general, however, it may be suboptimal for recovering a sparse signal, and the regularization problem $(QP_{1}^{\lambda})$ leads to a biased estimation by shrinking all the components
of the vector toward zero simultaneously, and sometimes results in over-penalization in the regularization model $(QP_{1}^{\lambda})$ as the $\ell_{1}$-norm in linear compressed sensing.

Inspired the good performance of the fraction function in image restoration and linear compressed sensing (see, e.g., \cite{geman14, li30}), in this paper, we replace $\|x\|_{0}$ with a continuous sparsity
promoting penalty function
\begin{equation}\label{equ8}
P(x)=P_{a}(x)=\sum_{i=1}^{n}\rho_{a}(x_{i}),\ \ \ a>0
\end{equation}
where
\begin{equation}\label{equ9}
\rho_{a}(t)=\frac{a|t|}{a|t|+1}
\end{equation}
is the fraction function, and it is called "strictly non-interpolating" in [14]. Clearly, $\rho_{a}(t)$ is increasing and concave in $t\in[0,+\infty]$.

With the change of parameter $a>0$, the non-convex function $P_{a}(x)$ interpolates the $\ell_{0}$-norm
\begin{equation}\label{equ10}
\lim_{a\rightarrow+\infty}\rho_{a}(x_{i})=\left\{
    \begin{array}{ll}
      0, & {\ \ \mathrm{if} \ x_{i}=0;} \\
      1, & {\ \ \mathrm{if} \ x_{i}\neq 0.}
    \end{array}
  \right.
\end{equation}

\begin{figure}[h!]
 \centering
 \includegraphics[width=2.7in]{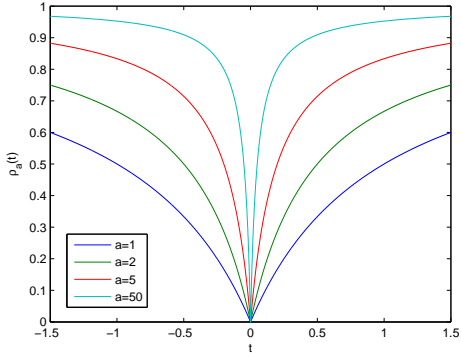}
%figure caption is below the figure
\caption{The behavior of the fraction function $\rho_{a}(t)$ for various values of $a>0$.}
\label{fig1}       % Give a unique label
\end{figure}

By this transformation, the minimization problem $(QP_{0})$ could be translated into the following minimization problem
\begin{equation}\label{equ11}
(QP_{a})\ \ \ \ \ \min_{x\in \mathcal{R}^{n}} P_{a}(x)\ \ \mathrm{subject}\ \mathrm{to}\ \ F(x)x=b
\end{equation}
for the constrained problem and
\begin{equation}\label{equ12}
(QP_{a}^{\lambda})\ \ \ \ \ \min_{x\in \mathcal{R}^{n}}\Big\{\|F(x)x-b\|_{2}^{2}+\lambda P_{a}(x)\Big\}.
\end{equation}
for the regularization problem.

The rest of this paper is organized as follows. In Section \ref{section2}, we propose an iterative fraction thresholding algorithm to solve the regularization problem $(QP_{a}^{\lambda})$ for all $a>0$. In
Section \ref{section3}, we present some numerical experiments to demonstrate the effectiveness of our algorithm. The concluding remarks are presented in Section \ref{section4}.

\section{Iterative fraction thresholding algorithm (IFTA)}\label{section2}
In this section, the iterative fraction thresholding algorithm (IFTA) is proposed to solve the regularization problem $(QP_{a}^{\lambda})$ for all $a>0$. Before we embark to this discussion, some
results need to be expressed before IFTA is proposed to solve the regularization problem $(QP_{a}^{\lambda})$.\\

We define a function of $\beta\in \mathcal{R}$ as
\begin{equation}\label{equ13}
f_{\lambda}(\beta)=(\beta-\gamma)^{2}+\lambda\cdot\rho_{a}(\beta)
\end{equation}
and
\begin{equation}\label{equ14}
\mathrm{prox}_{a,\lambda}^{\beta}(\gamma):=\arg\min_{\beta\in \mathcal{R}}f_{\lambda}(\beta).
\end{equation}

\begin{lemma}\label{lem1}
The operator $\mathrm{prox}_{a,\lambda}^{\beta}$ defined in (\ref{equ14}) can be expressed as
\begin{equation}\label{equ15}
\mathrm{prox}_{a,\lambda}^{\beta}(\gamma)=\left\{
    \begin{array}{ll}
      g_{a,\lambda}(\gamma), & \ \ \mathrm{if} \ {|\gamma|> t^{\ast};} \\
      0, & \ \ \mathrm{if} \ {|\gamma|\leq t^{\ast}.}
    \end{array}
  \right.
\end{equation}
where $g_{a,\lambda}(\gamma)$ is defined as
\begin{equation}\label{equ16}
g_{a,\lambda}(\gamma)=sign(\gamma)(\frac{\frac{1+a|\gamma|}{3}(1+2\cos(\frac{\phi(\gamma)}{3}-\frac{\pi}{3}))-1}{a}),
\end{equation}
$$\phi(\gamma)=\arccos(\frac{27\lambda a^{2}}{4(1+a|\gamma|)^{3}}-1)$$
and the threshold value satisfies
\begin{equation}\label{equ17}
t_{a,\lambda}^{\ast}=\left\{
    \begin{array}{ll}
      t_{a,\lambda}^{1}, & \ \ \mathrm{if} \ {\lambda\leq \frac{1}{a^{2}};} \\
      t_{a,\lambda}^{2}, & \ \ \mathrm{if} \ {\lambda>\frac{1}{a^{2}}.}
    \end{array}
  \right.
\end{equation}
where
\begin{equation}\label{equ18}
t_{a,\lambda}^{1}=\frac{\lambda}{2}a, \ \ \ \ t_{a,\lambda}^{2}=\sqrt{\lambda}-\frac{1}{2a}.
\end{equation}
\end{lemma}
The proof of Lemma \ref{lem1} used the Cartan¡¯s root-finding formula expressed in terms of hyperbolic functions and it is  a special case of the reference \cite{xing27}, and the
detailed proof can be seen in \cite{li30}.

\begin{figure}[h!]
  \begin{minipage}[t]{0.49\linewidth}
  \centering
  \includegraphics[width=1.1\textwidth]{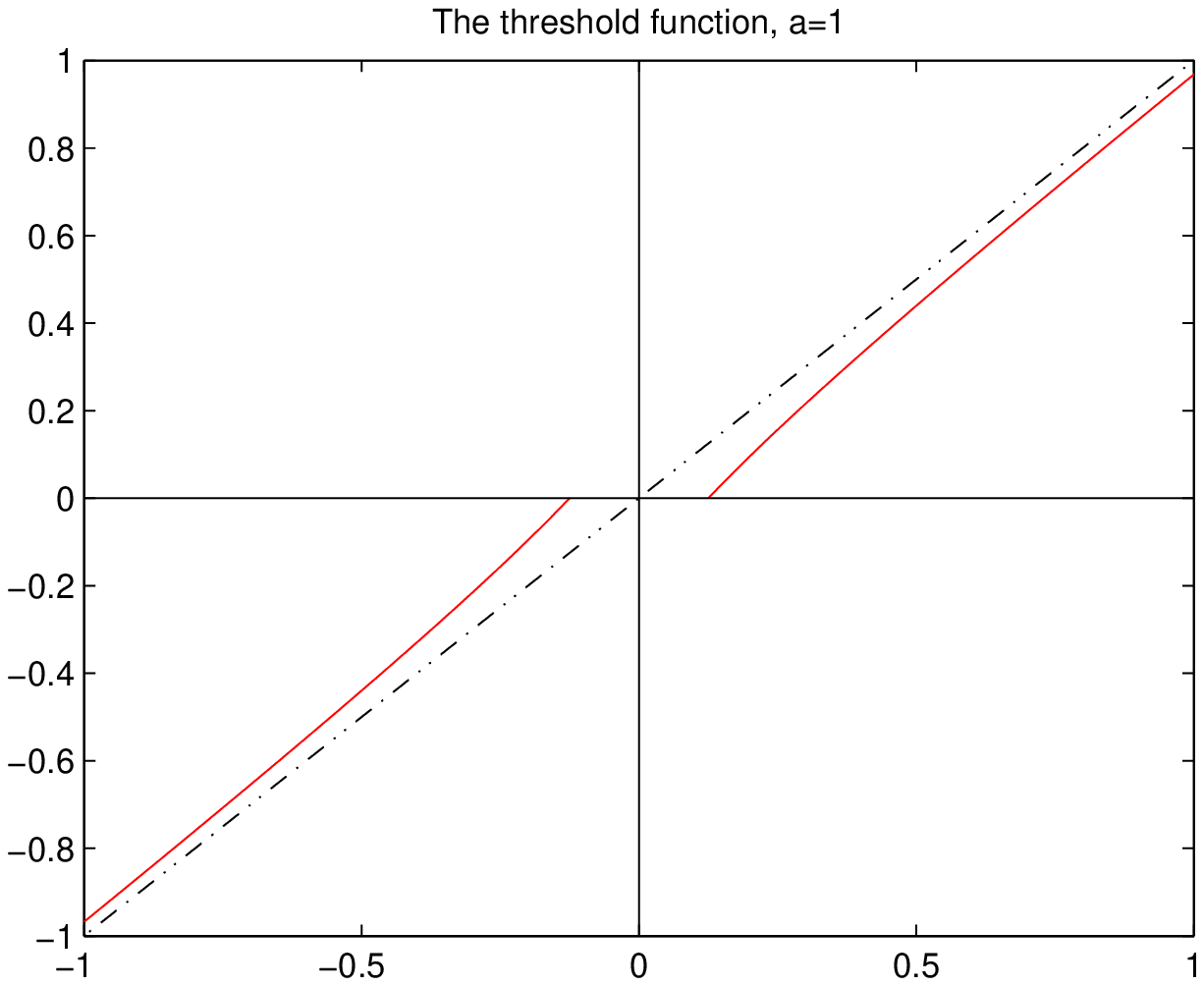}
  \end{minipage}
  \begin{minipage}[t]{0.49\linewidth}
  \centering
  \includegraphics[width=1.1\textwidth]{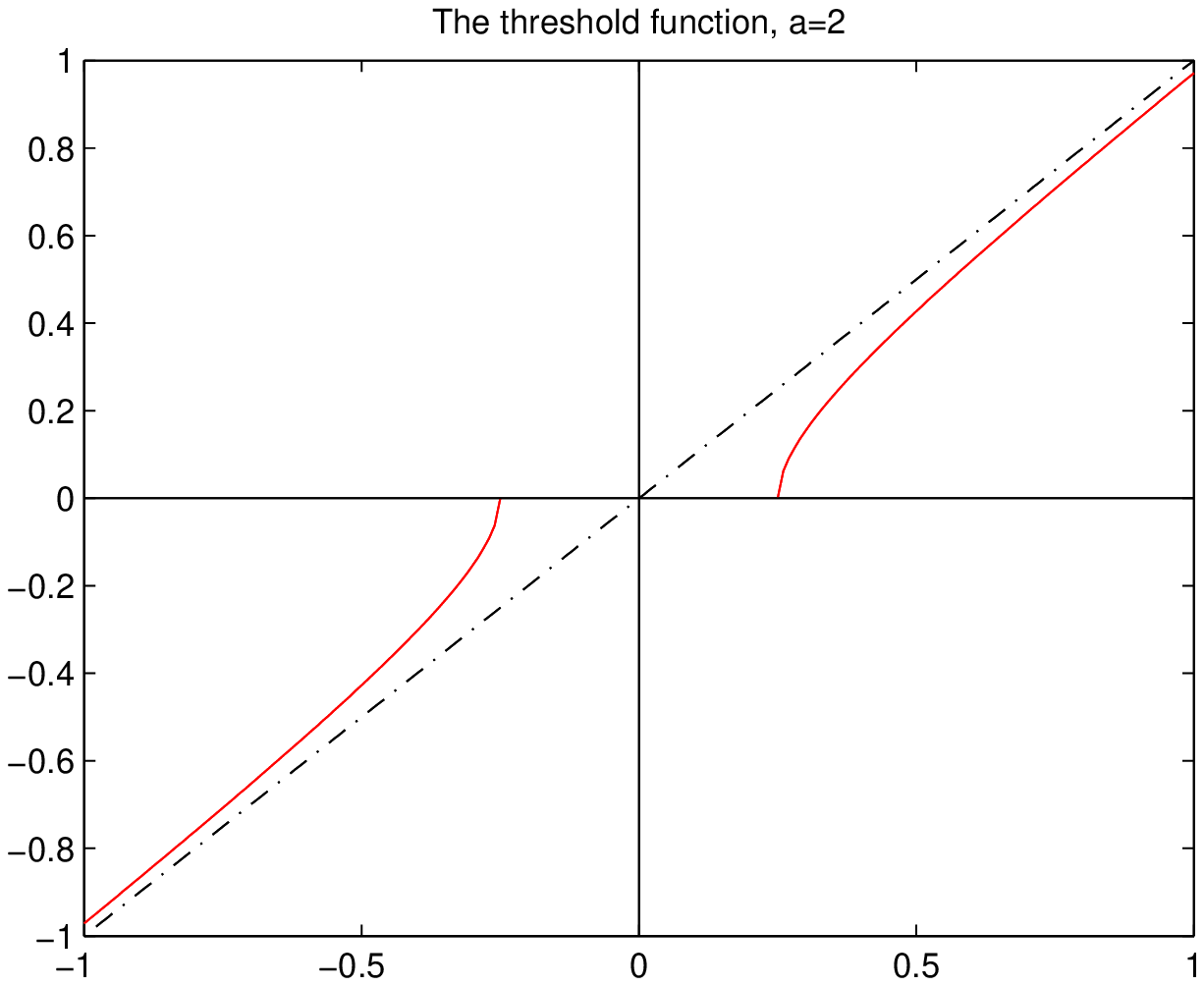}
  \end{minipage}
   \begin{minipage}[t]{0.49\linewidth}
  \centering
  \includegraphics[width=1.1\textwidth]{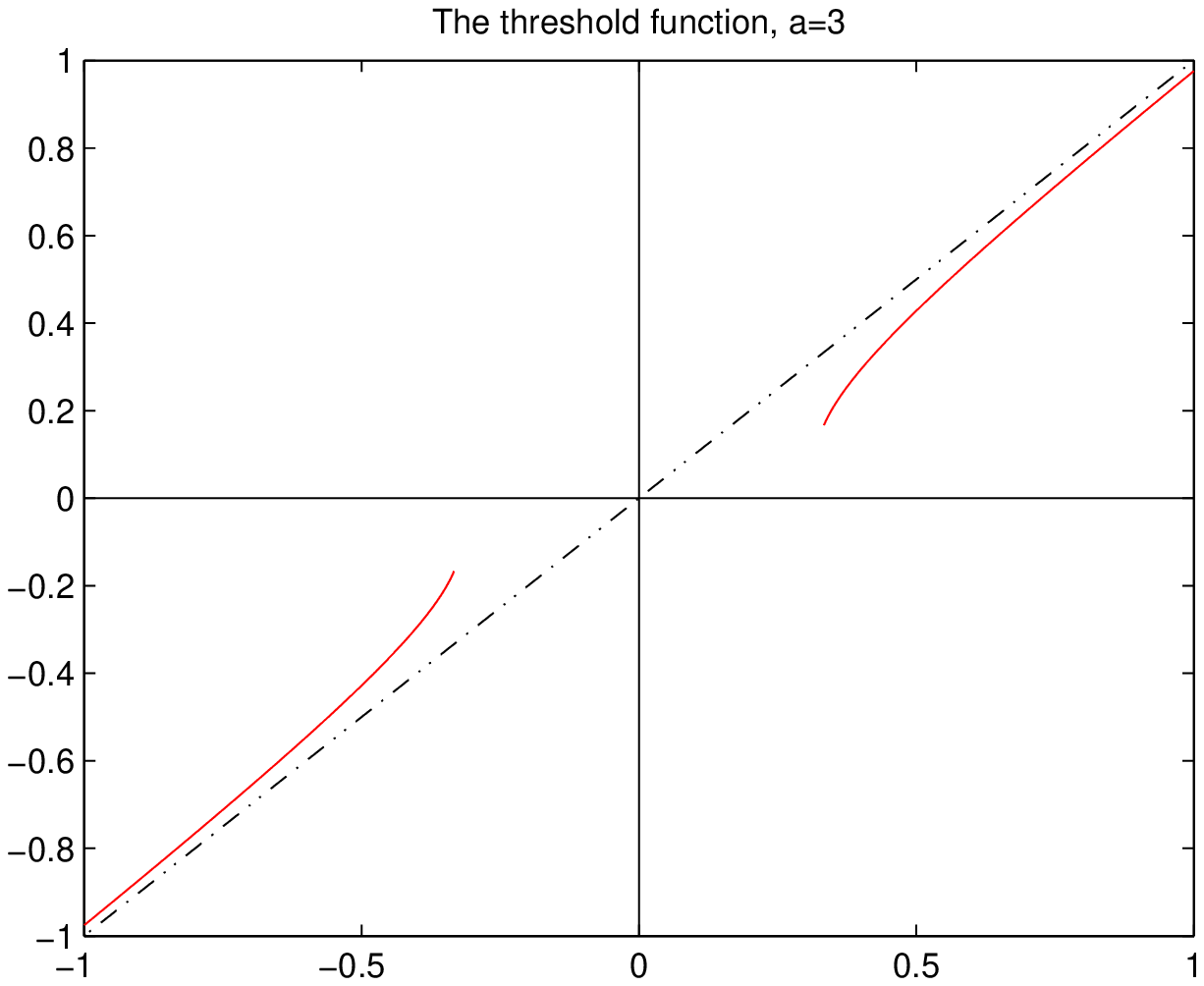}
  \end{minipage}
   \begin{minipage}[t]{0.49\linewidth}
  \centering
  \includegraphics[width=1.1\textwidth]{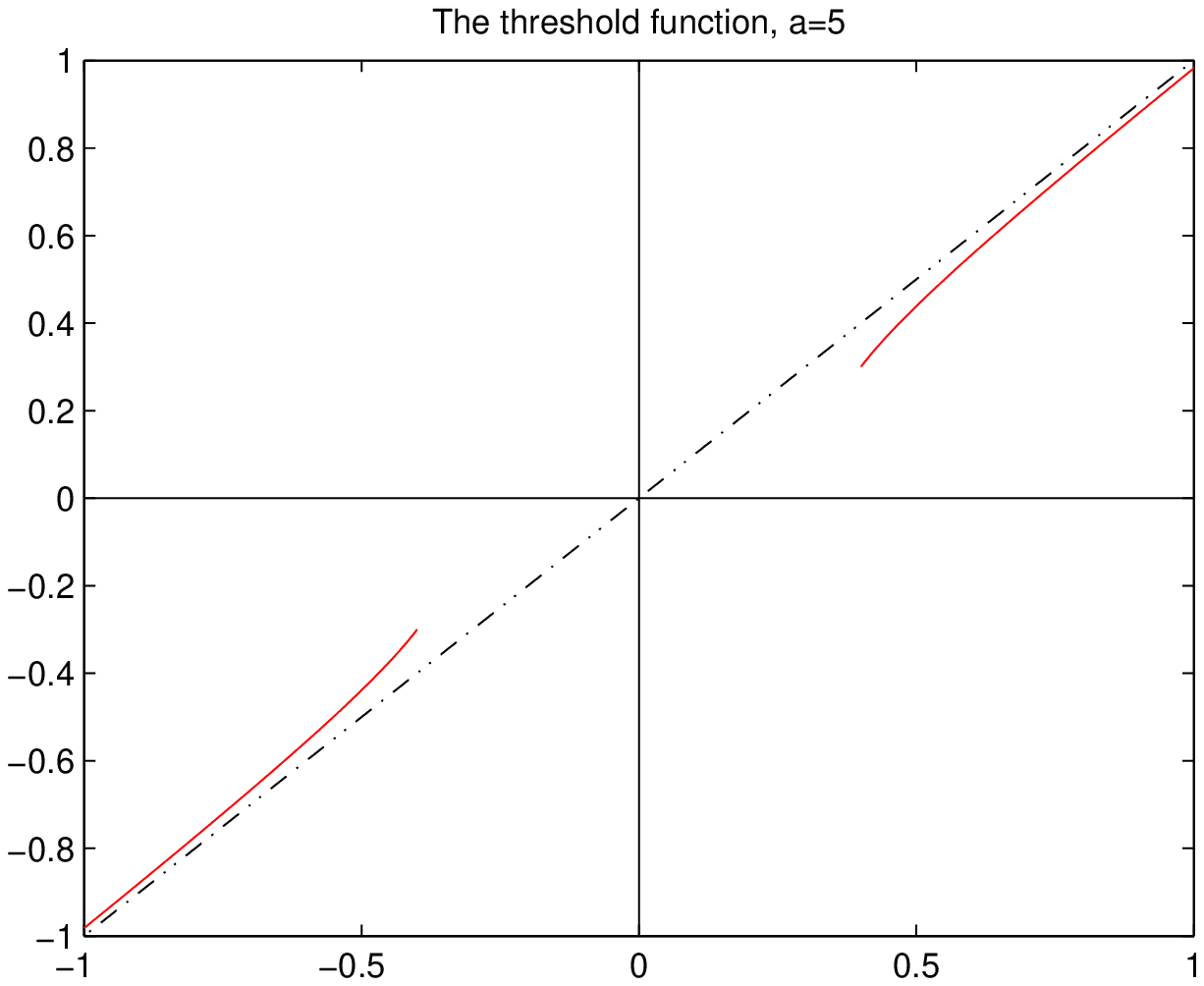}
  \end{minipage}
  \caption{The plots of the threshold function $g_{a,\lambda}$ for a=1, 2, 3, 5, and $\lambda=0.25$.} \label{fig2}
\end{figure}

\begin{definition}\label{de1}
The iterative thresholding operator $G_{\lambda, P}$ can be defined by
\begin{equation}\label{equ19}
G_{\lambda, P}(x)=(\mathrm{prox}_{a,\lambda}^{\beta}(x_{1}), \cdots, \mathrm{prox}_{a,\lambda}^{\beta}(x_{n}))^{T}
\end{equation}
where $\mathrm{prox}_{a,\lambda}^{\beta}$ is defined in Lemma \ref{lem1}.
\end{definition}

Nextly, we will show that the optimal solution to $(QP_{a}^{\lambda})$ could be expressed as a operation.

For any fixed positive parameters $\lambda>0$, $\mu>0$, $a>0$ and $x\in \Re^{n}$, let
\begin{equation}\label{equ20}
C_{1}(x)=\|F(x)x-b\|_{2}^{2}+\lambda P_{a}(x)
\end{equation}
and
\begin{equation}\label{equ21}
\begin{array}{llll}
C_{2}(x, y)&=&\mu\|F(y)x-b\|_{2}^{2}+\lambda\mu P_{a}(x)\\
&&-\mu\|F(y)x-F(y)y\|_{2}^{2}+\|x-y\|_{2}^{2}.
\end{array}
\end{equation}
Clearly, $C_{2}(x,x)=\mu C_{1}(x)$.
\begin{theorem}\label{the1}
For any fixed positive parameters $\lambda>0$, $\mu>0$ and $y\in \Re^{n}$, $\displaystyle\min_{x\in \Re^{n}}C_{2}(x,y)$
is equivalent to
\begin{equation}\label{equ22}
\min_{x\in \Re^{n}}\Big\{\|x-B_{\mu}(y)\|_{2}^{2}+\lambda\mu P_{a}(x)\Big\}
\end{equation}
where $B_{\mu}(y)=y+\mu F(y)^{\ast}(b-F(y)y)$.
\end{theorem}
\begin{proof}
By the definition, $C_{2}(x,y)$ can be rewritten as
\begin{eqnarray*}
C_{2}(x,y)&=&\|x-(y-\mu F(y)^{\ast}F(y)y+\mu F(y)^{\ast}b)\|_{2}^{2}\\
&&+\lambda\mu P_{a}(x)+\mu\|b\|_{2}^{2}+\|y\|_{2}^{2}-\mu\|F(y)y\|_{2}^{2}\\
&&-\|y-\mu F(y)^{\ast}F(y)y+\mu F(y)^{\ast}b\|_{2}^{2}\\
&=&\|x-B_{\mu}(y)\|_{2}^{2}+\lambda\mu P_{a}(x)+\mu\|b\|_{2}^{2}+\|y\|_{2}^{2}\\
&&-\mu\|F(y)y\|_{2}^{2}-\|B_{\mu}(y)\|_{2}^{2}
\end{eqnarray*}
which implies that $\displaystyle\min_{x\in \Re^{n}}C_{2}(x,y)$ for any fixed positive parameters $\lambda>0$, $\mu>0$ and
$y\in \mathcal{R}^{n}$ is equivalent to
$$\min_{x\in \Re^{n}}\Big\{\|x-B_{\mu}(y)\|_{2}^{2}+\lambda\mu P_{a}(x)\Big\}. $$
\end{proof}

\begin{theorem}\label{the2}
For any fixed positive parameter $\lambda>0$ and $0<\mu<L_{\ast}^{-1}$ with $\|F(x^{\ast})x-F(x^{\ast})x^{\ast}\|_{2}^{2}\leq L_{\ast}\|x-x^{\ast}\|_{2}^{2}$.
If $x^{\ast}$ is the optimal solution of $\displaystyle\min_{x\in \Re^{n}}C_{1}(x)$, then $x^{\ast}$ is also the optimal solution of
$\displaystyle\min_{x\in \Re^{n}}C_{2}(x,x^{\ast})$, that is
$$C_{2}(x^{\ast},x^{\ast})\leq C_{2}(x,x^{\ast})$$
for any $x\in \Re^{n}$.
\end{theorem}
\begin{proof}
By the definition of $C_{2}(x, y)$, we have
\begin{eqnarray*}
C_{2}(x,x^{\ast})&=&\mu\|F(x^{\ast})x-b\|_{2}^{2}+\lambda\mu P_{a}(x)\\
&&-\mu\|F(x^{\ast})x-F(x^{\ast})x^{\ast}\|_{2}^{2}+\|x-x^{\ast}\|_{2}^{2}\\
&\geq&\mu\|F(x^{\ast})x-b\|_{2}^{2}+\lambda\mu P_{a}(x)\\
&\geq&\mu C_{1}(x^{\ast})\\
&=&C_{2}(x^{\ast},x^{\ast}).
\end{eqnarray*}
\end{proof}

Theorem 2 shows that $x^{\ast}$ is the optimal solution of $\displaystyle\min_{x\in \Re^{n}}C_{2}(x,x^{\ast})$ as long as $x^{\ast}$ solves $\displaystyle\min_{x\in \Re^{n}}C_{1}(x)$. Moreover,
combined with Lemma \ref{lem1} and Theorem \ref{the1}, we can immediately conclude that the thresholding representation of $(QP_{a}^{\lambda})$ can be given by
\begin{equation}\label{equ23}
x^{\ast}=G_{a,\lambda\mu}(B_{\mu}(x^{\ast}))
\end{equation}
where the thresholding operator $G_{a,\lambda\mu}$ is obtained in Definition \ref{de1} by replacing $\lambda$ with $\lambda\mu$.

\begin{corollary}\label{co1}
For any fixed $\lambda>0$, $\mu>0$ and vector $x^{\ast}\in \Re^{n}$, let $x^{\ast}=G_{\lambda\mu, P}(B_{\mu}(x^{\ast}))$, then
$$
x^{\ast}_{i}=\left\{
    \begin{array}{ll}
      g_{a,\lambda\mu}(|B_{\mu}(x^{\ast})_{i}|), & \ \ \mathrm{if} \ {|B_{\mu}(x^{\ast})_{i}|> t_{a,\lambda\mu}^{\ast};} \\
      0, & \ \ \mathrm{if} \ {|B_{\mu}(x^{\ast})_{i}|\leq t_{a,\lambda\mu}^{\ast}.}
    \end{array}
  \right.
$$
where the threshold value $t_{a,\lambda\mu}^{\ast}$ is obtained in $(\ref{equ17})$ by replacing $\lambda$ with $\lambda\mu$.
\end{corollary}

With the thresholding representations (\ref{equ23}), the IFTA for solving the regularization problem $(QP_{a}^{\lambda})$ can be naturally defined as
\begin{equation}\label{equ24}
x^{k+1}=G_{\lambda\mu, P}(B_{\mu}(x^{k}))
\end{equation}
where $B_{\mu}(x^{k})=x^{k}+\mu F(x^{k})^{\ast}(b-F(x^{k})x^{k})$.

It is fairly well known that the quantity of the solution of a regularization problem depends seriously on the setting of the regularization
parameter. Here, the cross-validation method is accepted to select the proper regularization parameter. Nevertheless, when some prior information
is known for a regularization problem, this selection is more reasonable and intelligent. When doing so, the IFTA will be adaptive and free from
the choice of the regularization parameter.

To make this selection clear, we suppose that the vector $x^{\ast}$ of sparsity $r$ is the optimal solution of the regularization problem $(QP_{a}^{\lambda})$,
without loss of generality, we suppose that
\begin{eqnarray*}
&&|B_{\mu}(x^{\ast})|_{1}\geq|B_{\mu}(x^{\ast})|_{2}\geq\cdots\geq|(B_{\mu}(x^{\ast})|_{r}\\
&&\geq|(B_{\mu}(x^{\ast})|_{r+1}\geq\cdots\geq|(B_{\mu}(x^{\ast})|_{n}\geq0.
\end{eqnarray*}
By Corollary 1, the following inequalities hold
$$|B_{\mu}(x^{\ast})|_{i}>t_{a,\lambda\mu}^{\ast}\Leftrightarrow i\in\{1,2,\cdots,r\},$$
$$|B_{\mu}(x^{\ast})|_{i}\leq t_{a,\lambda\mu}^{\ast}\Leftrightarrow i\in\{r+1,r+2,\cdots,n\}.$$
According to $t_{a,\lambda\mu}^{2}\leq t_{a,\lambda\mu}^{1}$, we have
\begin{equation}\label{equ25}
\left\{
  \begin{array}{ll}
   |B_{\mu}(x^{\ast})|_{r}\geq t_{a,\lambda\mu}^{\ast}\geq t_{a,\lambda\mu}^{2}=\sqrt{\lambda\mu}-\frac{1}{2a}; \\
   |B_{\mu}(x^{\ast})|_{r+1}<t_{a,\lambda\mu}^{\ast}\leq t_{a,\lambda\mu}^{1}=\frac{\lambda\mu}{2}a,
  \end{array}
\right.
\end{equation}
which implies
\begin{equation}\label{equ26}
\frac{2|B_{\mu}(x^{\ast})|_{r+1}}{a\mu}\leq\lambda\leq\frac{(2a|B_{\mu}(x^{\ast})|_{r}+1)^{2}}{4a^{2}\mu}.
\end{equation}

Above estimation helps to set the optimal regularization parameter. For convenience, we denote by $\lambda_{1}$ and $\lambda_{2}$ the left and the right of above inequality respectively.
$$
\left\{
  \begin{array}{ll}
   \lambda_{1}=\frac{2|B_{\mu}(x^{\ast})|_{r+1}}{a\mu}; \\
   \lambda_{2}=\frac{(2a|B_{\mu}(x^{\ast})|_{r}+1)^{2}}{4a^{2}\mu},
  \end{array}
\right.
$$
A choice of $\lambda$ is
$$\lambda=\left\{
            \begin{array}{ll}
              \lambda_{1}, & \ \ {\mathrm{if}\ \lambda_{1}\leq\frac{1}{a^{2}\mu};} \\
              \lambda_{2}, &\ \ {\mathrm{if}\ \lambda_{1}>\frac{1}{a^{2}\mu}.}
            \end{array}
          \right.
$$
In practice, we approximate $B_{\mu}(x^{\ast})_{i}$ by $B_{\mu}(x^{k})_{i}$ in (\ref{equ27}), and we can take
\begin{equation}\label{equ27}
\begin{array}{llll}
\lambda=\left\{
            \begin{array}{ll}
              \lambda_{1,k}=\frac{2|B_{\mu}(x^{k})|_{r+1}}{a\mu},  & \ \ {\mathrm{if}\ \lambda_{1,k}\leq\frac{1}{a^{2}\mu};} \\
              \lambda_{2,k}=\frac{(2a|B_{\mu}(x^{k})|_{r}+1)^{2}}{4a^{2}\mu},  & \ \ {\mathrm{if}\ \lambda_{1,k}>\frac{1}{a^{2}\mu}.}
            \end{array}
          \right.
\end{array}
\end{equation}
in applications.

One more thing needs to be mentioned here is that the threshold value
\begin{equation}\label{equ28}
\begin{array}{llll}
t_{a,\lambda\mu}^{\ast}=\left\{
            \begin{array}{ll}
              \frac{\lambda\mu}{2}a,  & \ \ {\mathrm{if}\ \lambda=\lambda_{1,k};} \\
              \sqrt{\lambda\mu}-\frac{1}{2a},  & \ \ {\mathrm{if}\ \lambda=\lambda_{2,k}.}
            \end{array}
          \right.
\end{array}
\end{equation}
Notice that (\ref{equ27}) is valid for any $\mu>0$ satisfying $0<\mu\leq\|F(x_{k})\|_{2}^{-2}$. In general, we can take $\mu=\mu_{k}=\frac{1-\epsilon}{\|F(x_{k})\|_{2}^{2}}$
with any small $\epsilon\in(0,1)$ below.

\begin{algorithm}[h!]
\caption{: IFTA}
\label{alg:A}
\begin{algorithmic}
\STATE {Initialize: Given $x^{0}\in \mathcal{R}^{n}$, $\mu_{0}=\frac{1-\epsilon}{\|F(x^{0})\|_{2}^{2}}$ $(0<\epsilon<1)$ and $a>0$;}
\STATE {\textbf{while} not converged \textbf{do}}
\STATE \ \ \ \ \ {$z^{k}:=B_{\mu_{k}}(x^{k})=x^{k}+\mu_{k} F(x^{k})^{\ast}(y-F(x^{k})x^{k})$;}
\STATE \ \ \ \ \ {$\lambda_{1,k}=\frac{2|B_{\mu_{k}}(x^{k})|_{r+1}}{a\mu_{k}}$, $\lambda_{2,k}=\frac{(2a|B_{\mu_{k}}(x^{k})|_{r}+1)^{2}}{4a^{2}\mu_{k}}$, \\
\ \ \ \ \ $\mu_{k}=\frac{1-\epsilon}{\|F(x^{k})\|_{2}^{2}}$;}
\STATE \ \ \ {if\ $\lambda_{1,k}\leq\frac{1}{a^{2}\mu_{k}}$\ then}
\STATE \ \ \ \ \ \ {$\lambda=\lambda_{1,k}$; $t^{\ast}=\frac{\lambda\mu_{k} a}{2}$}
\STATE \ \ \ \ \ \ {for\ $i=1:\mathrm{length}(x)$}
\STATE \ \ \ \ \ \ {1.\ $|z^{k}_{i}|>t_{a,\lambda\mu_{k}}^{\ast}$, then $x^{k+1}_{i}=g_{a,\lambda_{k}\mu_{k}}(z^{k}_{i})$;}
\STATE \ \ \ \ \ \ {2.\ $|z^{k}_{i}|\leq t_{a,\lambda\mu_{k}}^{\ast}$, then $x^{k+1}_{i}=0$;}
\STATE \ \ \ {else}
\STATE \ \ \ \ \ \ {$\lambda=\lambda_{2,k}$; $t^{\ast}=\sqrt{\lambda\mu_{k}}-\frac{1}{2a}$;}
\STATE \ \ \ \ \ \ {for\ $i=1:\mathrm{length}(x)$}
\STATE \ \ \ \ \ \ {1.\ $|z^{k}_{i}|>t_{a,\lambda\mu_{k}}^{\ast}$, then $x^{k+1}_{i}=g_{a,\lambda_{k}\mu_{k}}(z^{k}_{i})$;}
\STATE \ \ \ \ \ \ {2.\ $|z^{k}_{i}|\leq t_{a,\lambda\mu_{k}}^{\ast}$, then $x^{k+1}_{i}=0$;}
\STATE \ \ \ {end}
\STATE \ \ \ \ \ \ \ {$k\rightarrow k+1$;}
\STATE{\textbf{end while}}
\STATE{\textbf{return}: $x^{k+1}$.}
\end{algorithmic}
\end{algorithm}

\section{Numerical experiments}\label{section3}
In the section, we carry out a series of simulations to demonstrate the performance of IFTA, and compare them with those obtained with some state-of-art methods (iterative soft thresholding algorithm (ISTA)[6,17],
iterative hard thresholding algorithm (IHTA)[6,17]). For each experiment, we repeatedly perform 30 tests and present average results and take $a=1$.

In our numerical experiments, we set
\begin{equation}\label{equ29}
F(x)=A_{1}+\eta f(\|x-x_{0}\|_{2}) A_{2}
\end{equation}
where $A_{1}\in\Re^{30\times 100}$ is a fixed Gaussian random matrix, $x_{0}\in \Re^{100}$ is a reference vector, $f:[0,\infty)\rightarrow \mathcal{R}$ is a positive and smooth Lipschitz continuous function
with $f(t)=\ln(t+1)$, $\eta$ is a sufficiently small scaling factor (we set $\eta=0.003$), and $A_{2}\in \Re^{30\times 100}$ is a fixed matrix with every entry equals to 1. Then the form of
nonlinearity considered in (\ref{equ29}) is a quasi-linear, and the more detailed accounts of the setting in form (\ref{equ29}) can be seen in [6,17]. By randomly generating such sufficiently sparse vectors
$x_{0}$ (choosing the non-zero locations uniformly over the support in random, and their values from $N(0,1))$, we generate vectors $b$. By this way, we know the sparsest solution to $F(x_{0})x_{0} = b$, and
we are able to compare this with algorithmic results.

The stopping criterion is usually as following
$$\frac{\|x_{k}-x_{k-1}\|_{2}}{\|x_{k}\|_{2}}\leq \mathrm{Tol}$$
where $x_{k}$ and $x_{k-1}$ are numerical results from two continuous iterative steps and $\mathrm{Tol}$ is a given small number. The success is measured by the computing
$$\mathrm{relative\ error}=\frac{\|x^{\ast}-x_{0}\|_{2}}{\|x_{0}\|_{2}}\leq \mathrm{Re}$$
where $x^{\ast}$ is the numerical results generated by IFTA, and $\mathrm{Re}$ is also a given small number. In all of our experiments, we set $\mathrm{Tol}=10^{-8}$ to indicate the stopping criterion, and set $\mathrm{Re}=10^{-4}$ to indicate a perfect recovery of the original sparse vector $x_{0}$.

\begin{figure}[h!]
 \centering
 \includegraphics[width=2.7in]{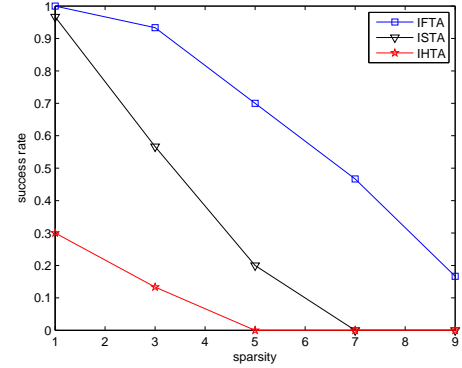}
%figure caption is below the figure
\caption{The success rate of three algorithms in the recovery of a sparse signal with different cardinality with $a=1$.}
\label{fig3}       % Give a unique label
\end{figure}

\begin{figure}[h!]
 \centering
 \includegraphics[width=2.7in]{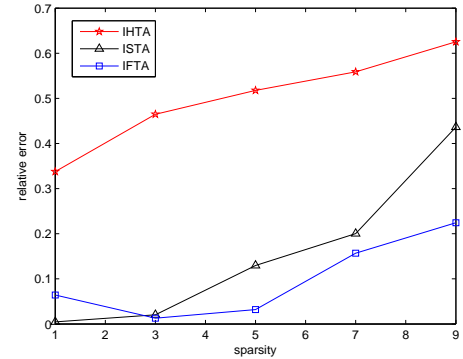}
%figure caption is below the figure
\caption{The relative error between the solution $x^{\ast}$ and the given signal $x_{0}$ with $a=1$.}
\label{fig4}       % Give a unique label
\end{figure}

The graphs presented in Fig.\ref{fig3} and Fig.\ref{fig4} show the performance of the ISTA, IHTA and IFTA in recovering the true (sparsest) signals. From Fig.3, we can see that IFTA performs best,
and IST algorithm the second. From Fig.\ref{fig4}, we can get that the IFTA always has the smallest relative error value with sparsity growing.

\section{Conclusion}\label{section4}
In the paper, we take the fraction function as the substitution for $\ell_{0}$-norm in quasi-linear compressed sensing. An iterative fraction thresholding algorithm is proposed to
solve the regularization problem $(QP_{a}^{\lambda})$ for all $a>0$. With the change of parameter $a>0$, our algorithm could get a promising result, which is one of the advantages for our algorithm 
compared with other algorithms. We also provide a series of experiments to assess performance of our algorithm, and the experiment results show that, compared with some state-of-art algorithms, our 
algorithm performs the best in the sparse signal recovery. However, the convergence of our algorithm is not proved theoretically in this paper, and it is our future work.

%\section*{Acknowledgment}
%The work was supported by the National Natural Science Foundations of China (11771347,11131006,41390450,11761003,11271297) and the Science Foundations of Shaanxi Province of China (2016JQ1029,2015JM1012).

\ifCLASSOPTIONcaptionsoff
  \newpage
\fi


\begin{thebibliography}{1}
\bibitem{candes1}
E. Candes, J. Romberg, T. Tao. Stable signal recovery from incomplete and inaccurate measurements.
Communications on Pure and Applied Mathematics, 59(8): 1207-1223 (2006)

\bibitem{dono2}
D. L. Donoho. Compressed sensing. IEEE Transaction on Information Theory, 52(4): 1289-1306 (2006)
% etc

\bibitem{bru3}
A. M. Bruckstein, D. L. Donoho, and M. Elad. From sparse solutions of systems of equations to sparse modelling
of signals and images. SIAM Review, 51(1): 34-81 (2009)

\bibitem{elad4}
M. Elad. Sparse and Redundant Representations: from Theory to Applications in Signal and Image Processing. Springe, New York, 2010.

\bibitem{the5}
S. Theodoridis, Y. Kopsinis, and K. Slavakis. Sparsity-aware learning and compressed sensing: an overview.
https://arxiv.org/pdf/1211.5231.

\bibitem{ehler6}
M. Ehler, M. Fornasier, and J. Sigl. Quasi-linear compressed sensing. Multiscale Modeling and Simulation, 12(2): 725-754 (2014)

\bibitem{sigl7}
J. Sigl. Quasilinear compressed sensing, Master's thesis, Technische University M\"{u}nchen, Munich, Germay, 2013.

\bibitem{candes8}
E. Candes, T. Tao. Decoding by linear programming. IEEE Transactions on Information Theory, 51(12): 4203-4215 (2005).

\bibitem{donoho9}
D. L. Dohoho, X. Huo. Uncertainty principles and ideal atomic decomposition. IEEE Transactions on Information Theory, 47(7): 2845-2862 (2001).

\bibitem{donoho10}
D. L. Donoho, J. Tanner. Sparse nonnegative solution of underdetermined linear equations by linear programming. Proceedings of the National Academy of Sciences of the United States of America, 102(27): 9446-9451 (2005).

\bibitem{donoho11}
D. L. Donoho, M. Elad. Optimally sparse representation in general (nonorthoganal) dictionaries via $\ell_{1}$ minimization. Proceedings of the National Academy of Sciences of the United States of America, 100(5): 2197-2202 (2003).

\bibitem{grib12}
R. Gribonval, M. Nielson. Sparse representations in unions of bases. IEEE Transactions on Information Theory, 49(12): 3320-3325 (2003).

\bibitem{chen13}
S. Chen, D. L. Donoho, and M. A. Saunders. Atomic decomosition by basic pursuit. SIAM Journal of Science Computing, 20(1): 33-61 (1999).

\bibitem{geman14}
D. Geman and G. Reynolds. Constrained restoration and recovery of discontinuities. IEEE Transactions on Pattern Analysis and Machine Intelligence, 14(3): 367-383 (1992)

\bibitem{15}
M. Nikolova. Local strong homogeneity of a regularized estimator. SIAM Journal on Applied Mathematics, 61(2): 633-658 (2000)

%\bibitem{16}
%H. Li, Q. Zhang, A. Cui and J. Peng. Minimization of fraction function penalty in compressed sensing. arXiv:1705.06048v1 [math.OC] 17 May 2017

\bibitem{17}
S. Foucart, H. Rauhut. A mathematic introduction to compressive sensing. Springe, New York, 2013.

\bibitem{li30}
H. Li, Q. Zhang, A. Cui, and J. Peng. Minimization of fraction function penalty in compressed sensing. https://arxiv.org/pdf/1705.06048.

\bibitem{xing27}
F. Xing, Investigation on solutions of cubic equations with one unknown. Journal of the Central University for Nationalities (Natural Sciences Edition), 12(3): 207-218 (2003)




\end{thebibliography}
\end{document}